\documentclass[a4paper,reqno,11pt]{amsart}
\usepackage{vmargin}
\usepackage{graphicx}
\usepackage{amsfonts}
\usepackage{amssymb}
\usepackage{amsrefs} 
\usepackage{mathrsfs}
\usepackage{stmaryrd}
\usepackage{amsmath}
\usepackage{amsthm}
\usepackage[shortlabels]{enumitem}
\usepackage{nomencl}
\usepackage[bookmarks=true]{hyperref}   
\usepackage{esint}
\usepackage{dsfont}
\usepackage{upgreek}
\usepackage{xcolor}
\usepackage{slashed}
\usepackage{scalerel}
\usepackage{comment}
\usepackage{todonotes}
\usepackage{skull}
\usepackage{tikz-cd}
\usepackage{faktor}
\usepackage{mathtools}
\usepackage{nicefrac}
\hypersetup{
    colorlinks,%
}

\author{Lashi Bandara}
\title
{The relative index theorem for general first-order elliptic operators}
\date{\today}

\address{Lashi Bandara, 
Mathematics Department, 
Brunel University London, 
Kingston Lane, Uxbridge, Middlesex,  UB8 3PH
}
\urladdr{\href{https://www.brunel.ac.uk/people/menaka-bandara}{https://www.brunel.ac.uk/people/menaka-bandara}}
\email{\href{mailto:lashi.bandara@brunel.ac.uk}{lashi.bandara@brunel.ac.uk}}

\newcommand{\cE}{\mathcal{E}}
\newcommand{\cF}{\mathcal{F}}

\newcommand{\cK}{\mathcal{K}} 
\newcommand{\cM}{\mathcal{M}} 
\newcommand{\cN}{\mathcal{N}}

\newcommand{\mg}{\mathrm{g}}
\newcommand{\mh}{\mathrm{h}}

\newcommand{\Dir}{{\rm D} }
\newcommand{\Ad}{{\rm A} }

\newcommand{\sym}{\upsigma}

\newcommand{\checkH}{\check{\mathrm{H}}}	
\newcommand{\hatH}{\hat{\mathrm{H}}}		

\newcommand*{\coker}{\operatorname{coker}}

\newcommand*{\ind}{\operatorname{ind}}

\makeatletter
\@namedef{subjclassname@2020}{%
  \textup{2020} Mathematics Subject Classification}
\makeatother


\def\colour{\colour}

\def\colour{\color}

\newtheorem{theorem}{Theorem}[section]
\newtheorem{corollary}[theorem]{Corollary}
\newtheorem{lemma}[theorem]{Lemma}
\newtheorem{proposition}[theorem]{Proposition}

\newtheorem{definition}[theorem]{Definition}

\newtheorem{remark}[theorem]{Remark}

\newcommand{\cbrac}[1]{\left(#1\right)}
\newcommand{\bbrac}[1]{\left[#1\right]}
\newcommand{\dbrac}[1]{\left\{#1\right\}}

\newcommand{\set}[1]{\dbrac{#1}}

\newcommand{\dom}{\mathrm{dom}}
\newcommand{\ran}{\mathrm{ran}}

\newcommand{\e}{\mathrm{e}}

\newcommand{\Co}{\mathbb{C}}
 
\newcommand{\In}{\mathbb{Z}}

\renewcommand{\emptyset}{\varnothing}
\newcommand{\union}{\cup}

\newcommand{\disunion}{\sqcup}

\newcommand{\rest}[1]{{{\lvert_{}}_{}}_{#1}}
\newcommand{\close}[1]{\overline{#1}}		

\newcommand{\id}{{\rm id}}
\renewcommand{\epsilon}{\varepsilon}

\renewcommand{\phi}{\varphi}

\newcommand{\graph}{\mathrm{graph}}		

\newcommand{\norm}[1]{\| #1 \|}			

\newcommand{\spt}[1]{{\rm spt} {\text{ }}#1}	

\DeclareMathOperator{\tr}{tr}			

\newcommand{\interior}[1]{\mathring{#1}}	
\newcommand{\cotanb}{{\rm T}^\ast}	

\DeclareFontFamily{OT1}{restrictfont}{}
\DeclareFontShape{OT1}{restrictfont}{m}{n}{<-> fmvr8x}{}

\newcommand{\inprod}[1]{\left\langle #1 \right\rangle}	
\newcommand{\sinprod}[1]{\langle #1 \rangle}	

\newcommand{\bddlf}{\mathcal{L}} 	

\newcommand{\Lp}[2][{}]{{\rm L}^{#2}_{\rm #1}}		
\newcommand{\Ck}[2][{}]{{\rm C}^{#2}_{\rm #1}}		
\newcommand{\Hard}[2][{}]{{\rm H}^{#2}_{\rm #1}}		
\newcommand{\SobH}[2][{}]{\Hard[#1]{\rm #2}}

\newcommand{\Dmax}{\Dir_{\max}}
\newcommand{\Dmin}{\Dir_{\min}}
\newcommand{\Hil}{\mathcal{H}} 
\newcommand{\Baps}{\mathrm{B_{APS}}}
\newcommand{\Bmatch}{\mathrm{B_{M}}}

\begin{document}

\maketitle
\begin{abstract}
The relative index theorem is proved for general first-order elliptic operators  that are complete and coercive at infinity  over measured manifolds.
This extends the original result by Gromov-Lawson for generalised Dirac operators as well as the result  of Bär-Ballmann for Dirac-type operators. 
The theorem is seen through the point of view of boundary value problems, using the graphical decomposition of elliptically regular boundary conditions for general first-order elliptic operators due to Bär-Bandara.
Splitting, decomposition and the Phi-relative index theorem are proved on route to the relative index theorem. 
\end{abstract} 
\tableofcontents

\parindent0cm
\setlength{\parskip}{\baselineskip}

\section{Introduction}

In \cite{GL} Gromov and Lawson proved the famed relative index theorem for Spin-Dirac operators on a class of non-compact manifolds. 
Their primary motivation was to use this result for the study of positive scalar curvature metrics closed manifolds. 
Their method was to study cylinder like regions over this closed manifold and utilise the Weitzenböck identity for the Spin-Dirac operator, where scalar curvature emerges as the lower order term relating the connection Laplacian to the Spin Laplacian.

Since then, this theorem has been generalised by multiple authors and in a myriad of directions.
Although an exhaustive list is too numerous to provide here, we refer to \cite{Bunke} by Bunke and references therein for an historical account of the topic.  

A generalisation of particular relevance for our purposes was obtained by Bär-Ballmann in   \cite{BB12}.
A significant aspect of this generalisation was their approach to relative index theory from the point of view of boundary value problems.
In particular,  a certain class of boundary conditions, namely those that are elliptically regular, were characterised through an associated graphical decomposition.
This allowed for the ability to deform, in a controlled manner, elliptically regular boundary conditions to the Atiyah-Patodi-Singer boundary condition, where this latter condition arose  famously in the series of papers \cite{APS0, APS1, APS2, APS3}.
However, the general framework in \cite{BB12} is restricted to a special class of first-order elliptic operators, of which Dirac-type operators are the quintessential example.
 
In \cite{BBan}, Bär and the author studied boundary value problems for general first-order elliptic operators.
This was largely motivated by the desire to study the Rarita-Schwinger operator on $\nicefrac 32$-spinors, a physically determined operator arising naturally in geometric contexts. 
It is an operator which fails to be of Dirac-type and is perhaps the quintessential example of a non Dirac-type operator.
In \cite{BBan}, the authors generalised the equivalence between elliptic regularity for a boundary condition and the admission of a graphical decomposition.
This set the stage to consider relative index theorems for operators beyond those that are Dirac-type.
Consequently, in this paper, we prove the following theorem, generalising the relative index theorem of Gromov-Lawson in \cite{GL} to the fullest extent.

\begin{theorem}[Relative index theorem]
\label{Thm:RelInd}
Let $(\cM_{1},\mu_1)$ and $(\cM_{2},\mu_2)$ be measured manifolds without boundary, $\Dir_{i}:\Ck{\infty}(\cM_i,\cE_i)\to\Ck{\infty}(\cM_i,\cF_i) $ first-order elliptic and complete (c.f. Definition~\ref{Def:Complete}). 

Suppose that: 
\begin{enumerate}[label=(\roman*)]
\item $\Dir_{1},\Dir_{2}$ agree outside compact $\cK_{1},\cK_{2}$ (the bundles and operators pull back in a canonical way, c.f. Definition~\ref{Def:Agree}); 
\item  $\mu_{1}=f^{\ast}\mu_{2}$ on $\cM_1 \setminus \cK_1$;
\item there is a compact two-sided hypersurface $\cN_{1}\subset\cM_{1}$ separating $\cM_1$  such that $\cM_{1}=\cM_{1}^{\prime}\cup\cM_{1}^{\prime\prime}$ with $\partial\cM_{1}^{\prime}=\partial\cM_{1}^{\prime\prime}=\cN_1$ (i.e., cutting along $\cN_1$ decomposes $\cM_1$ into two manifolds with boundary, c.f. Definition~\ref{Def:Cutting} and Definition~\ref{Def:Separate})  with $\cK_{1}\subset\interior{\cM}_{1}$.
\end{enumerate}

Then $\Dir_{1}$ is Fredholm  if and only if  $\Dir_{2}$ is Fredholm and in that case
\[
\ind(\Dir_{1}) -\ind(\Dir_{2}) =\int_{\cK_{1}}\alpha_{0,\Dir_{1}}\ d\mu_1 -\int_{\cK_{2}}\alpha_{0,\Dir_{2}}\ d\mu_2.
\]
The term  $\alpha_{0,\Dir_{i}} = \alpha_{0,\tilde{\Dir}_{i}}\rest{\cK_i}$ is the constant term appearing in the asymptotic expansion as $t \to 0$ of the kernel
$$ \tr \cbrac{\e^{-t \tilde{\Dir}_{i}^\ast \tilde{\Dir}_{i}} - \e^{-t\tilde{\Dir}_{i} \tilde{\Dir}_{i}^\ast}} (x) 
	\sim \sum_{k \geq - n} t^{\frac{k}{2}} \alpha_{k,\tilde{\Dir_{i}}}(x),$$
where $\tilde{\Dir}_{i}$ is an extension of $\Dir_{i}$ on an larger closed manifold containing $\cK_{i}$.
\end{theorem}

\section*{Acknowledgements}
The author was supported by SPP2026 from the German Research Foundation (DFG).
Penelope Gehring and Peter Grabs are duly acknowledged for their help in the delivery of a graduate level course at the University of Potsdam from which this paper emerged. 
The author also wishes to thank Christian Bär for useful discussions, encouragement and support.
Last but not least, the anonymous referee is acknowledged for their helpful comments and suggestions. 

\section{Preliminaries}

Let $(\cM,\mu)$ be a connected measured manifold with compact boundary $\partial\cM \subset \cM$ and  $(\cE,\mh^E), (\cF,\mh^F)\to\cM$ two Hermitian vector bundles over $\cM$.
The function spaces $\Ck{\infty}(\cM,\cE)$, $\Ck[c]{\infty}(\cM,\cF)$  and $\Ck[cc]{\infty}(\cM,\cE)$ respectively denote smooth sections, compactly supported smooth sections (which are allowed to touch the boundary) and compactly supported smooth sections supported away from the boundary. 

Furthermore, let $\Dir:\Ck{\infty}(\cM,\cE)\to\Ck{\infty}(\cM,\cF)$  be a first-order elliptic differential operator.
Recall that in this case, there exists a unique formal adjoint $\Dir^\dagger:\Ck{\infty}(\cM,\cF)\to\Ck{\infty}(\cM,\cE)$ to $\Dir$.
That is, 
$$\sinprod{\Dir u,v}_{\Lp{2}(\cM,\cF)} =\sinprod{u,\Dir^\dagger v}_{\Lp{2}(\cM,\cE)}$$ 
where $u\in\Ck[cc]{\infty}(\cM,\cE)$ and $v\in\Ck[cc]{\infty}(\cM,\cF)$.
The maximal operator corresponding to $\Dir$ is given by $\Dir_{\max} := (\Dir^\dagger\rest{\Ck[cc]{\infty}})^\ast$, where $\ast$ denotes the $\Lp{2}$-adjoint.
The minimal extension is $\Dir_{\min} :=\close{\Dir\rest{\Ck[cc]{\infty}}}$. 
By construction, it is clear that $\Dir_{\min} \subset \Dir_{\max}$.

In order to understand boundary conditions purely from data on the boundary, we need to ensure that potential ``implicit'' boundaries at infinity, arising from incompleteness, are ruled out.
To capture this, we require the following notion.
\begin{definition}[Complete]
\label{Def:Complete}
The operators $\Dir$ is said to be \emph{complete} if the subspace $\set{u \in \dom(\Dir_{\max}): \spt u\ \text{compact in }\cM}$ is dense in $\dom(\Dir_{\max})$.
\end{definition}

Note that in this definition, since our convention is $\partial \cM \subset \cM$, the dense subspace consists of sections that are allowed to touch the boundary.
This automatically self-improves in regularity: for complete $\Dir$,  we obtain $\Ck[c]{\infty}(\cM, \cE)$ is dense in $\dom(\Dir_{\max})$.
See Theorem~2.3~(i) in \cite{BBan}.

Let us now assume that $\Dir$ and $\Dir^\dagger$ are complete operators.
This ensures that $\faktor{\dom(\Dir_{\max})}{\dom(\Dir_{\min})}$ can be controlled purely in terms of a certain function space on the boundary. 
We will see that this function space is described via the class of gadgets on the boundary captured in the following definition.

\begin{definition}[Adapted boundary operator]
A first-order differential operator $\Ad:\Ck{\infty}(\partial\cM,\cE)\to\Ck{\infty}(\partial\cM,\cE)$ is called \emph{an adapted boundary operator (to $\Dir$)} if there exists  an inward pointing covectorfield $\tau \in \Ck{\infty}(\partial \cM, \cotanb \cM)$ such that the principal symbol of $\Ad$  satisfies:
\begin{equation} 
\label{Eq:AdBd} 
\sym_{\Ad}(x,\xi) =\sym_{\Dir}(x,\tau)^{-1}\circ\sym_{\Dir}(x,\xi), 
\end{equation} 
for all $x\in\partial\cM$ and $0\neq\xi\in\cotanb\cM$.
\end{definition}

In \cite{BBan}, two important assertions are made. 
The first  is that for any given inward pointing $\tau \in \Ck{\infty}(\partial \cM, \cotanb\cM)$, there exists an adapted boundary operator $\Ad$ satisfying \eqref{Eq:AdBd}. 
Moreover, such an $\Ad$ can always be chosen invertible $\omega$-bisectorial (the spectrum sits in a bisector in the complex plane containing the real line, and there are resolvent estimates of the form $\norm{(\zeta - A)^{-1}}\leq C$ outside of this bisector).
Moreover, the spectral cuts $\chi^{-}(\Ad)$ and $\chi^{+}(\Ad)$, respectively projecting to the generalised eigenspaces to the left and right of the imaginary axis, exist as pseudo-differential operators of order zero.
The second assertion is that, on defining 
\begin{equation} 
\label{Eq:Czech} 
\checkH_{\Ad}(\Dir)  :=\chi^-(\Ad)\SobH{\frac 12}(\partial\cM,\cE)\oplus\chi^+(\Ad)\SobH{-\frac12} (\partial\cM,\cE),
\end{equation}
and assuming that $\Dir$ and $\Dir^\dagger$ are complete, we obtain that the boundary trace map $\gamma: \Ck[c]{\infty}(\cM,\cE) \to \Ck{\infty}(\partial \cM,\cE)$ given by $u\mapsto u\rest{\partial\cM}$ extends uniquely to a bounded surjection $\gamma:\dom(\Dir_{\max})\to\checkH_{\Ad}(\Dir)$ with $\ker\gamma =\dom(\Dir_{\min})$.
See Proposition~4.6 and  Theorem~2.3 in \cite{BBan}. 

A  particular consequence of these results is that 
\begin{equation}
\label{Eq:BdyIso} 
\faktor{\dom(\Dir_{\max})}{\dom(\Dir_{\min})}\cong\checkH_{\Ad}(\Dir),
\end{equation} 
where the isomorphism  is the canonical quotient map induced from the boundary trace map $\gamma$.
Note that space $\checkH_{\Ad}(\Dir)$ is fixed as a set - it is the range of the boundary trace map $\gamma$.
An adapted boundary operator $\Ad$ only determines a particular way to compute its topology.
This justifies writing $\checkH(\Dir)$ in place of $\checkH_{\Ad}(\Dir)$, where the latter captures the particular norm for the topology given in terms of $\Ad$.

From a geometric point of view, the particular operator $\Ad$ can be thought of as analogous to a coordinate system. 
There are many choices for $\Ad$, where some choices may be better in computations for a given problem than others.
If the completeness assumption on $\Dir$ or $\Dir^\dagger$ is discarded, then  ``incompleteness'' away from the boundary can render the Banach space isomorphism in  \eqref{Eq:BdyIso} to fail.

The isomorphism \eqref{Eq:BdyIso} illustrates that all extensions of $\Dir_{\min}$ contained in  $\Dir_{\max}$ can be identified uniquely to a subspace of $\checkH(\Dir)$.
An extension is closed if and only if the subspace corresponding to this operator in $\checkH(\Dir)$ is closed.
\emph{Boundary conditions} are precisely closed subspaces of $\checkH(\Dir)$ and the space $\checkH(\Dir)$  is the total space of boundary conditions for $\Dir$.
The adjoint boundary condition is given by $B^\dagger =\set{v\rest{\partial\cM}: v\in\dom(\Dir_{B}^\ast)}$ and  $\Dir_B^\ast = \Dir^\dagger_{B^\dagger}$.

Keeping the isomorphism \eqref{Eq:BdyIso} in mind, note that the subspace $B_+ :=\chi^+(\Ad)\SobH{-\frac 12}(\partial\cM,\cE)$ is a perfectly legitimate boundary condition.
However, it is easy to see from this that $\dom(\Dir_{B_+})\not\subset\SobH[loc]{1}(\cM,\cE)$.
Therefore, unlike the situation for closed manifolds, we do not automatically have that extensions of an elliptic first-order operator $\Dir$ are $\SobH[loc]{1}$-regular.
When a boundary condition $B\subset\checkH_{\Ad}(\Dir)$ further satisfies $B\subset\SobH{\frac12}(\partial\cM,\cE)$, then we say that $B$ is \emph{elliptically semi-regular}.
If both  $B$ and $B^\dagger$ are  elliptically semiregular, then we say that $B$ is \emph{elliptically regular}. 
These are the quintessential class of boundary condition that will be used in this paper.

While this qualitative description of elliptically regular boundary conditions is conceptually appealing, a tangible characterisation is required to work with them.
To that end, we introduce the following notion.

\begin{definition} [Graphical $\Lp{2}$-decomposition]
\label{Def:GraphDecomp}
\index{Graphical $\Lp{2}$-decomposition}
Let $B\subset\checkH ( \Dir)$ and $\Ad$ an invertible  bisectorial adapted boundary operator to $\Dir$.  
Suppose that:

\begin{enumerate}[label=(\roman*)]
\item
There exist mutually complementary subspaces $W_{\pm}$ and $V_{\pm}$ of $\Lp{2}(\partial\cM,\cE) $ such that
\[
W_{\pm}\oplus V_{\pm}=\chi^{\pm}(\Ad)\Lp{2}(\partial\cM,\cE) .
\]

\item
$W_{\pm},W_{\pm}^{\ast}\subset\SobH{\frac{1}{2}}(\partial\cM,\cE) $ and are finite dimensional.

\item
There exists a bounded linear map
\[
g: V_{-}\to V_{+}
\]
such that 
\begin{align*}
g( V_{-}\cap\SobH{\frac{1}{2}}(\partial\cM,\cE) ) &\subset V_{+}\cap\SobH{\frac{1}{2}}(\partial\cM,\cE) ,
\\
g^{\ast}( V_{+}^{\ast}\cap\SobH{\frac{1}{2}}(\partial\cM,\cE) ) &\subset V_{-}^{\ast}\cap\SobH{\frac{1}{2}}(\partial\cM,\cE) ,
\end{align*}
and
\begin{align*}
B 
& =\set{ v+gv: v\in V\_\cap\SobH{\frac{1}{2}}(\partial\cM,\cE) }\oplus W_{+}.
\end{align*}
\end{enumerate}
Then we say that $B$ is \emph{$\Lp{2}$-graphically decomposable} with respect to\ $\chi^+(\Ad)$.
\end{definition}

A key result from \cite{BBan}, generalising the results of \cite{BB12}, is the following. 
\begin{theorem}[Theorem~2.9 \cite{BBan}]
\label{Thm:EllDecomp}
A subspace  $B$ is an elliptically regular boundary condition, i.e., $B\subset\SobH{\frac12}(\partial\cM,\cE)$ and $B^\dagger\subset\SobH{\frac12}(\partial\cM,\cF)$ if and only if for some (and hence all) invertible bisectorial adapted boundary operators $\Ad$,  $B$ is  $\Lp{2}$-graphically decomposable with respect to $\chi^{+}(\Ad)$.
\end{theorem} 

Motivated by Atiyah-Patodi-Singer \cite{APS1} and in light of the developments in \cite{BBan}, for a given invertible bisectorial adapted boundary operator $\Ad$, its associated  \emph{Atiyah-Patodi-Singer (APS)} boundary condition is defined as
$$\Baps(\Ad) :=\chi^-(\Ad)\SobH{\frac12}(\partial\cM,\cE).$$
From Theorem~\ref{Thm:EllDecomp}, it is easy to see that this is an elliptically regular boundary condition.

\section{Coercivity and index} 

It is of import for us to be able to understand when extensions of operators are Fredholm. 
Much like the notion of completeness for an operator ensures the lack of potential ``boundary'' near infinity, the following condition ensures a sense of ``Fredholmness'' near infinity.

\begin{definition}
\label{Def:Coercive}
The operator  $\Dir$ is said to be \emph{coercive at infinity} if there exists $C > 0$ and a compact $\cK\subset\cM$ such that 
$$\norm{u}_{\Lp{2}(\cM,\cE)}\leq C\norm{\Dir u}_{\Lp{2}(\cM,\cF)}.$$
for all $u\in\Ck{\infty}(\cM,\cE)$  such that $\spt u \subset\cM\setminus \cK$.
\end{definition}

The following proposition provides a method in which to detect the coercivity at infinity of an operator, through the use of elliptically regular boundary conditions.

\begin{proposition}\label{Prop:CoerFred}
Let $B$ is semi-elliptically regular, i.e.\ $B$ is a boundary condition and $B\subset\SobH{\frac{1}{2}}(\partial \cM,\cE) $.
Then $\Dir$ is coercive at infinity if and only if $\Dir_{B}$ has finite dimensional kernel and closed range.
\end{proposition}
\begin{proof}
The proof mirrors the proof of Theorem~8.5 in \cite{BB12}. 

The key idea in the proof of this, in the ``only if'' direction, is to use the elliptic semi-regularity of $B$ to assert that $\dom(\Dir_{B})\subset\SobH[loc]{1}(\cM,\cE)$.  
This ensures the estimate 
\begin{equation}
\label{Eq:CompactNorm}
\norm{u}_{\SobH{1}(\cK',\cE)}\lesssim\norm{\Dir_B (\chi u)}_{\Lp{2}(\cM,\cF)} +\norm{\chi u}_{\Lp{2}(\cM,\cE)}, 
\end{equation} 
whenever $u\in\dom(\Dir_B)$ and $\chi\in\Ck[c]{\infty}(\cM, [0,1])$ which is identically $1$ on $\cK'$, a compact set satisfying $\cK\subset \cK'$.
Since $\cK'$ is compact, $\SobH{1}(\cK',\cE)$ embeds compactly into $\Lp{2}(\cK',\cE)$. 
Given a bounded sequence $u_n$ in $\dom(\Dir_{B})$ such that $\Dir_{B} u_n \to v$, we obtain $u \in \Lp[loc]{2}(\cM,\cE)$ such that $u_n \to v$, possibly on passing to a subsequence.
Then, using \eqref{Eq:CompactNorm}, we obtain $u \in \Lp{2}(\cM,\cE)$. 
That is, we have shown that for a bounded sequence $u_n \subset \dom(\Dir_{B})$, there is a convergence subsequence $u_{n_k}$, which yields that $\Dir_{B}$ has finite dimensional kernel and closed range (c.f. Proposition~A.3 in \cite{BB12}).
\end{proof}

As an immediate consequence, we obtain the following Fredholmness result.

\begin{corollary}
\label{Cor:CoerFred}
If $\Dir,\Dir^{\dagger}$ are coercive at infinity, and $B$ is elliptically regular, then $\Dir_{B}$ is Fredholm and
\[
\ind(\Dir_{B}) =\dim\ker(\Dir_{B})  -\dim\ker(\Dir_{B^{\dagger}}^{\dagger}) \in\In.
\]
\end{corollary}

\section{Deformations of boundary conditions}

A virtue of being able to identify the entirety of boundary conditions using \eqref{Eq:Czech} and \eqref{Eq:BdyIso} is the ability to study perturbations and deformations of boundary conditions.
\begin{definition}
\label{Def:DeformBC}
A family of boundary conditions $B_{s}\subset\checkH (\Dir) $ for $s\in\left[ 0,1\right] $ is said to be a continuous deformation from $B_{0}$ to $B_{1}$, if there exist isomorphisms
\[
\phi_{s}: B_{0}\to B_{s}\text{ with }\phi_{0}=\id
\]
with $s\mapsto\phi_{s}\in\Ck{0}([0,1],\bddlf( B_{0},\checkH (\Dir) )).$ 
\end{definition}

\begin{remark}
\label{Rem:ContDeform}
In what is to follow, deformations of boundary conditions are paramount for elliptically regular boundary conditions. 
For such a boundary condition $B$, in addition to the fact that $B$ is closed in $\checkH(\Dir)$, we have that $B\subset\SobH{\frac12}(\cM,\cE)$ is closed.
Since $\SobH{\frac12}(\cM,\cE)\subset\checkH(\Dir)$, which in particular means we have the estimate:
\begin{equation} 
\label{Eq:SobEmbed} 
\norm{u}_{\checkH(\Dir)}^2 =\norm{\chi^+(\Ad)u}_{\SobH{-\frac12}(\partial\cM,\cE)}^2 +\norm{\chi^-(\Ad)u}_{\SobH{\frac12}(\partial\cM,\cE)}^2\lesssim \norm{u}_{\SobH{\frac12}(\partial\cM,\cE)}^2,
\end{equation} 
we and use Lemma A.3 in \cite{BBan} to assert that 
\begin{equation}
\label{Eq:SobIso} 
\norm{u}_{\checkH(\Dir)}\simeq\norm{u}_{\SobH{\frac12}(\partial\cM,\cE)}
\end{equation}
for all $u\in B$.
Since the implicit constant appearing in \eqref{Eq:SobEmbed} only depends on $\Ad$, the implicit constant in \eqref{Eq:SobIso} is independent of $B$.
Therefore, if we have $B_s\subset\SobH{\frac12}(\partial\cM,\cE)$ for all $s\in [0,1]$, then we obtain
$$s\mapsto\phi_s\in\Ck{0}([0,1],\SobH{\frac12}(\partial\cM,\cE)).$$
In particular, this is the case when $B_s$ is elliptically regular for all $s\in [0,1]$.
\end{remark} 

The following proposition describes the way in which the deformation of a boundary condition deforms the operator itself.
In the light of \eqref{Eq:BdyIso}, this is certainly to be expected.

\begin{proposition}\label{Prop:DeformIso}
Let $\Dir$ and $\Dir^\dagger$ be complete and suppose that $s\mapsto\phi_{s}$ is a continuous deformation of boundary conditions.
Then there exists a map $s\mapsto\Phi_{s}\in\bddlf(\dom(\Dir_{B_{0}}) ,\dom(\Dmax) ) $  such that  $\Phi_{s}\dom(\Dir_{B_{0}}) =\dom(\Dir_{B_{s}}) $.
\end{proposition}

\begin{proof}
Recall that $\faktor{\dom(\Dmax)}{\dom(\Dmin)}\cong\checkH (\Dir) $.
Since $\dom(\Dir_{B_s})\subset\dom(\Dmax) $ is a closed subspace and $\faktor{\dom(\Dir_{B_{s}})}{\dom(\Dmin)}\subset\faktor{\dom(\Dmax)}{\dom(\Dmin)}$ we obtain that
\[
\faktor{\dom(\Dir_{B_{s}})}{\dom(\Dmin)}\cong B_{s}
\]
in the sense of Banach spaces with the constant in the isomorphism independent of $B_{s}$.

Now,
\begin{align*}
\dom(\Dir_{B_{s}}) &\cong\faktor{\dom(\Dir_{B_{s}})}{\dom(\Dmin)}\oplus\dom(\Dmin)
\\
&\cong B_{s}\oplus\dom(\Dmin)
\\
&\cong\phi_{s}( B_{0})\oplus\dom(\Dmin)
\\
&\cong B_{0}\oplus\dom(\Dmin)
\\
&\cong\faktor{\dom(\Dir_{B_{0}})}{\dom(\Dmin)}\oplus\dom(\Dmin)
\\
&\cong\dom(\Dir_{B_{0}})\oplus\dom(\Dmin) .
\end{align*}
In the fourth isomorphism, $s\mapsto\phi_{s}$ is continuous and determines $\Phi_{s}:\dom(\Dir_{B_{0}})\to\dom(\Dir_{B_{s}}) $ continuously.
The conclusion follows.
\end{proof}

With the use of Remark~\ref{Rem:ContDeform}, we obtain the following corollary for elliptically regular boundary conditions.
\begin{corollary}
\label{Cor:DeformBC}
Let $\Dir$ and $\Dir^\dagger$ be complete and coercive at infinity. 
Let $\Ad$ be any invertible bisectorial adapted boundary operator and $B$ be an elliptically regular boundary condition. 
Write
\[
B=\graph( g\rest{\SobH{\frac{1}{2}}})\oplus W_{+}
\]
with respect to $\chi^+(\Ad)$ (see Definition~\ref{Def:GraphDecomp}).
Define
\begin{align*}
\phi_{s}: B_{0}=V_{-}\oplus W_{+} \to B_{s} &:=\graph( sg\rest{\SobH{\frac{1}{2}}})\oplus W_{+},
\\
\phi_s(v+w_{+}) &:= v+sgv+w_{+},
\end{align*}
where we recall $V_{\pm}\oplus W_{\pm}=\chi^{\pm}(\Ad)\Lp{2}(\partial M,E) $ from Definition~\ref{Def:GraphDecomp}.

Then $s\mapsto\phi_{s}: B_{0}\to\checkH(\Dir)$ is a continuous deformation of boundary conditions and
\[
\ind(\Dir_{B_{0}}) =\ind(\Dir_{B_{s}}) =\ind(\Dir_{B}) .
\]
\end{corollary}
\begin{proof}
It is immediate from construction that  $\phi_{s}: B_0\to B_s$ is an isomorphism with  $\phi_0 =\id$.
By what we have said in Remark \ref{Rem:ContDeform}, it is a continuous deformation of boundary conditions.

Let $\Phi_{s}:\dom(\Dir_{B_{0}})\to\dom(\Dir_{B_{s}}) $ be the induced isomorphism from Proposition~\ref{Prop:DeformIso}.
Then, we have that $\Dir_{B_{s}}\circ\Phi_{s}:\dom(\Dir_{B_{0}})\to\Lp{2}( \cM,\cE) $ is bounded, continuous in $s$, and therefore,
\[
\ind(\Dir_{B_{s}}\circ\Phi_{s}) =\ind( B_{0})
\]
since the index is invariant under a continuous deformation.
But since $\Phi_{s}$ is an isomorphism, the dimension of the kernel and cokernel remains unchanged, so therefore $\ind(\Dir_{B_{s}}) =\ind(\Dir_{B_{s}}\circ\Phi_{s}) $.
\end{proof}

\begin{lemma}\label{Prop:FredContain}
Let $\Dir$ and $\Dir^\dagger$ be complete and coercive at infinity.
Let $B_{1}\subset B_{2}$ be elliptically regular boundary conditions.
Then, $\dim \cbrac{ \faktor{B_{2}}{B_{1}}} <\infty$ and
\[
\ind(\Dir_{B_{2}}) =\ind(\Dir_{B_{1}}) +\dim\cbrac{\faktor{B_{2}}{B_{1}}} .
\]
\end{lemma}
\begin{proof}
The condition $B_{1}\subset B_{2}$ is equivalent to $\Dir_{B_{1}}\subset\Dir_{B_{2}}$.
Therefore, $\dom(\Dir_{B_{1}})\subset\dom(\Dir_{B_{2}}) $, $\ker(\Dir_{B_{1}})\subset\ker(\Dir_{B_{2}}) $ and $\ran(\Dir_{B_{1}})\subset\ran(\Dir_{B_{2}}) $.
Note that the ranges are closed because $B_{i}$ are Fredholm boundary conditions.

Since these are Hilbert spaces, we find orthogonal complements
\begin{align*}
\ker(\Dir_{B_{2}}) =\ker(\Dir_{B_{1}})\oplus ^{\perp}K\text{ and }\dom(\Dir_{B_{i}}) =\ker(\Dir_{B_{i}})\oplus^{\perp}R_{i} .
\end{align*}
Note that $R_{i}\cong\ran(\Dir_{B_{i}}) $ via $\Dir: R_{i}\to\ran(\Dir_{B_{i}}) $ and therefore, $R_1\subset R_2$ let
\[
R_{2} =R_{1}\oplus^{\perp}R.
\]
Therefore,
\begin{align*}
\dom(\Dir_{B_{2}}) & =\ker(\Dir_{B_{2}})\oplus R_{2} =\ker(\Dir_{B_{1}})\oplus K\oplus R_{1}\oplus R =\dom(\Dir_{B_{1}})\oplus K\oplus R.
\end{align*}
Now
\[
\faktor{B_{2}}{B_{1}}\cong\faktor{\dom(\Dir_{B_{2}})}{\dom(\Dir_{B_{1}})}\cong K\oplus R ,
\]
where the first isomorphism is readily verified, and the second follows from our construction above.

We prove that $\faktor{B_{2}}{B_{1}}$ is finite dimensional.
It suffices to prove that $K$ and $R$ are finite dimensional.
First, we note that $K$ is finite dimensional, since $\ker(\Dir_{B_{2}}) $ is finite dimensional by Fredholmness of $B_2$.

To show $R$ is finite dimensional, note
\begin{multline*}
\coker(\Dir_{B_{1}})\cong\faktor{\Lp{2}(\cM,\cE)}{\ran(\Dir_{B_{1}})}
\cong\faktor{\Lp{2}(\cM,\cE)}{\ran(\Dir_{B_{2}})}\oplus\faktor{\ran(\Dir_{B_{2}})}{\ran(\Dir_{B_{1}})}
\\
\cong\coker(\Dir_{B_{2}})\oplus\faktor{R_{2}}{R_{1}}
\cong\coker(\Dir_{B_{2}})\oplus R,
\end{multline*}
where the second isomorphism follows from the readily verifiable fact:
\[
\faktor{\Hil}{X}\cong\faktor{Y}{X}\oplus\faktor{\Hil}{Y}
\]
when $X\subset Y\subset\Hil$ are closed subspaces of a Hilbert space $\Hil$.
By the Fredholmness of $B_1$, $\coker(\Dir_{B_{1}})$ is finite dimensional and hence $R$ is finite dimensional.

Now we prove the index formula in the conclusion.
We simply calculate:
\begin{align*}
\ind(\Dir_{B_{2}}) & =\dim(\ker(\Dir_{B_{2}}) ) -\dim(\coker(\Dir_{B_{2}}) )
\\
& =\dim(\ker(\Dir_{B_{1}}) ) +\dim( K) -\dim(\coker(\Dir_{B_{1}}) ) +\dim( R)
\\
& =\ind(\Dir_{B_{1}}) +\dim( K\oplus R)
\\
& =\ind(\Dir_{B_{1}}) +\dim(\faktor{B_{2}}{B_{1}}) .
\qedhere
\end{align*}
\end{proof}

\begin{remark}
This argument only requires that $\Dir_{B_1}$ and $\Dir_{B_2}$ are Fredholm operators. 
Therefore, the assumption of elliptic regularity of $B_1$ and $B_2$ can be relaxed and replaced by the condition that $B_i$ are \emph{Fredholm boundary conditions}, by which we mean that $\Dir_{B_i}$ is a Fredholm operator. 
\end{remark} 

Combining these results, we obtain the following theorem of this section. 

\begin{theorem}
\label{Thm:DeformAps}
Let $\Dir$ and $\Dir^\dagger$ be complete and coercive at infinity. 
Fix to $\Ad$ be an adapted boundary operator and $B$ an elliptically regularly boundary condition. 
Let $W_{\pm}$ be the subspaces arising from the graphical decomposition of $B$ with respect to $\chi^+(\Ad)$ as given in Definition~\ref{Def:GraphDecomp}.
\[
\ind( \Dir_{B}) =\ind( \Dir_{\Baps(\Ad)}) +\dim( W_{+}) -\dim( W_{-}) ,
\]
\end{theorem}

\begin{proof}
For notational convenience, write $B_{-}:=\Baps(\Ad)$.
Using Theorem~\ref{Thm:EllDecomp}, write $B =\graph(g\rest{\SobH{\frac12}})\oplus W_+$,  where $\chi^\pm(\Ad)\checkH(\Dir) = V_\pm\oplus W_{\pm}$.
Let $B_{0}:=V_{-}\oplus W_{+}$ as in Corollary~\ref{Cor:DeformBC} and from there, we obtain
\[
\ind( \Dir_{B}) =\ind( \Dir_{B_{0}}) .
\]
Now let us consider $B_{-}\oplus W_{+}$, which is an elliptically regular boundary condition since $W_{+}\subset\SobH{\frac{1}{2}}$ is finite dimensional and $B_{-}=\chi^{-}(\Ad)\SobH{\frac{1}{2}}(\partial \cM,\cE) $. Since $B_{-}\oplus W_{+}\supset B_{0}$ and
\[
\faktor{B_{-}\oplus W_{+}}{B_{0}} =\faktor{( V_{-}\oplus W_{-}\oplus W_{+})}{( V_{-}\oplus W_{+})}\cong W_- ,
\]
using Proposition~\ref{Prop:FredContain},
\[
\ind( \Dir_{B_{-}\oplus W_{+}}) =\ind( \Dir_{B_{0}}) +\dim( W_{-}) .
\]
Also, $B_{-}\subset B_{-}\oplus W_{+}$ and therefore,
\[
\ind( \Dir_{B_{-}\oplus W_{+}}) =\ind( B_{-}) +\dim( W_{+}) .
\]
Combining these two equations, we get
\[
\ind( \Dir_{B_{0}}) +\dim( W_{-}) =\ind( \Dir_{B_{-}}) +\dim( W_{+})
\]
which is the formula appearing in the conclusion.
\end{proof}
\begin{remark}
The subspaces $W_{\pm}$ can be described explicitly as: 
$$ W_+ = B \cap \chi^+(\Ad)\SobH{\frac12} (\partial \cM,\cE)\quad\text{and}\quad W_- = \chi^-(\Ad)\bbrac{ B^{\perp, \hatH_{\Ad}(\Dir)} \cap \SobH{\frac12}(\partial \cM,\cE)},$$
where 
$$\hatH_{\Ad}(\Dir) = \chi^-(\Ad^\ast) \SobH{-\frac12}(\partial \cM,\cE) \oplus \chi^+(\Ad^\ast)\SobH{\frac12}(\partial\cM,\cE).$$
This space is isomorphic to the dual space of $\checkH(\Dir)$, which readily follows from the fact that $\inprod{\cdot,\cdot}: \checkH_{\Ad}(\Dir) \times \hatH_{\Ad}(\Dir) \to \Co$ is a perfect paring extending the $\Lp{2}(\partial \cM,\cE)$ inner product.

The $\hatH_{\Ad}(\Dir)$ space is canonically isomorphic to  $\checkH_{\tilde{\Ad}}(\Dir^\dagger)$ via $\sym_0 = \sym_{\Dir}(\cdot, \tau)$,  where $\tilde{\Ad} := - (\sym_0^{-1})^\ast \Ad \sym_0^\ast$ is the adapted operator for $\Dir^\dagger$ canonically determined from $\Ad$.
The advantage of $\hatH_{\Ad}(\Dir)$ over $\checkH_{\tilde{\Ad}}(\Dir^\dagger)$ is that it captures the adjoint problem over $\cE \to \partial \cM$ rather than over $\cF \to \partial \cM$. 
\end{remark}

\section{Splittings and decompositions}

We now consider the way in which to relate the index of an operator to two operators obtained by cutting a manifold along a compact two-sided hypersurface.
For that, throughout this section, assume that  $\cM^{\prime}$  is a connected manifold with $\partial\cM^{\prime} =\emptyset$.

\begin{definition}[Cutting along a hypersurface]
\label{Def:Cutting}
Let $\cN\subset\cM^{\prime}$ be a two-sided compact hypersurface in $\cM^{\prime}$ (i.e.\ $\cN$ has a trivial normal bundle).
Then by ``cutting along $\cN$'', we obtain the manifold with boundary
\[
\cM:=(\cM^{\prime}\setminus\cN)\cup(\cN_{1}\sqcup\cN_{2}) ,
\]
where $\cN_{1}=\cN$, $\cN_{2}=-\cN$ (i.e.\ with opposite orientation) and with $\partial\cM=\cN_{1}\sqcup\cN_{2}$.
\end{definition}

Given a density $\mu^{\prime}$ on $\cM^{\prime}$ and bundles $\cE',\cF'\to\cM^{\prime}$, there are the naturally and canonically induced objects $\mu,\cE,\cF$ via pullback to $\cM$.
If $\Dir':\Ck{\infty}(\cM^{\prime},\cE')\to\Ck{\infty}(\cM^{\prime},\cF')$, then it is clear that there is a naturally induced operator $\Dir:\Ck{\infty}(\cM,\cE)\to\Ck{\infty}(\cM,\cF)$.

\begin{proposition}\label{Prop:Split}
We have
\[
\Lp{2}(\partial\cM, E) =\Lp{2}(\cN_1,\cE)\oplus^\perp\Lp{2}(\cN_2,\cE) =\Lp{2}(\cN,\cE)\oplus^\perp\Lp{2}(\cN,\cE) .
\]
Suppose that $\Ad_{0}$ is an invertible bisectorial adapted boundary operator on $\cN_{1}=\cN$.
Then,$-\Ad_{0}$ is an invertible bisectorial adapted boundary operator on $\cN_{2}$ and $\Ad :=\Ad_{0}\oplus(-\Ad_{0})$ is an invertible bisectorial adapted boundary operator on $\partial\cM=\cN_{1}\sqcup\cN_{2}$.
Moreover,
\begin{align*}
\checkH (\Dir)
\cong &\cbrac{\chi^{-}(\Ad_{0})\SobH{\frac{1}{2}}(\cN,\cE)\oplus\chi^{+}(\Ad_{0})\SobH{-\frac{1}{2}}(\cN,\cE) ) } \\
&\qquad\qquad\oplus\cbrac{\chi^{+}(\Ad_{0})\SobH{\frac{1}{2}}(\cN,\cE)\oplus\chi^{-}(\Ad_{0})\SobH{-\frac{1}{2}}(\cN,\cE) }.
\end{align*}
\end{proposition}

\begin{proof}
The splitting of $\Lp{2}(\partial\cM,\cE)$ follows immediately from the fact that $\cM =\cN_1\disunion\cN_2$. 

For the splitting of $\checkH(\Dir)$, note that
\begin{align*}
\chi^-(\Ad)\SobH{\frac12}(\partial\cM)&\oplus\chi^+(\Ad)\SobH{-\frac12}(\partial\cM) \\
&=\cbrac{\chi^{-}(\Ad_{0})\SobH{\frac{1}{2}}(\cN_{1},\cE)\oplus\chi^{-}(-\Ad_{0})\SobH{\frac{1}{2}}(\cN_{2},\cE) }  \\
&\qquad\qquad\oplus\cbrac{\chi^{+}(\Ad_{0})\SobH{-\frac{1}{2}}(\cN_1,\cE)\oplus\chi^{+}( -\Ad_{0})\SobH{-\frac{1}{2}}(\cN_2,\cE) }\\
&\cong\cbrac{\chi^{-}(\Ad_{0})\SobH{\frac{1}{2}}(\cN,\cE)\oplus\chi^{+}(\Ad_{0})\SobH{-\frac{1}{2}}(\cN,\cE) } \\
&\qquad\qquad\oplus\cbrac{\chi^{+}(\Ad_{0})\SobH{\frac{1}{2}}(\cN,\cE)\oplus\chi^{-}(\Ad_{0})\SobH{-\frac{1}{2}}(\cN,\cE) }.
\qedhere
\end{align*}
\end{proof}

Since $\cN$ is compact, completeness and coercivity at infinity on $\cM$ for $\Dir$ is an inherited property from $\cM'$. 
The following is immediate from using Definition~\ref{Def:Coercive} along with the fact that $\cN$ is compact.

\begin{lemma}
\label{Lem:CoerInherited}
$\Dir,\Dir^{\dagger}$ are complete and coercive at infinity  if and only if  $\Dir'$ and $(\Dir') ^{\dagger}$ are.
\end{lemma}

In order to connect information regarding the operator $\Dir'$ on $\cM'$, the boundaryless manifold, and $\cM$ obtained from cutting along $\cN$, we define the following.

\begin{definition}[Matching condition]
\label{Def:Matching}
The subspace 
\[
\Bmatch:=\set{ (u,u) \in \SobH{\frac12}(\cN_1, \cE) \oplus \SobH{\frac12}(\cN_2, \cE) : u\in\SobH{\frac{1}{2}}(\cN,\cE) }  \subset \checkH(\Dir)
\]
is called the \emph{matching condition}.
\end{definition}

\begin{remark}
Note that $\SobH{\frac12}(\partial\cM,\cE)$ is a dense subspace of $\checkH(\Dir)$.
Therefore, it can never be a boundary condition (i.e. closed). 
However, since $\partial \cM =\cN_1\disunion\cN_2$, as we have already see in Proposition~\ref{Prop:Split}, 
\begin{multline*}
\chi^-(\Ad)\SobH{\frac12}(\partial\cM, \cE) 
=\chi^{-}(\Ad_{0})\SobH{\frac{1}{2}}(\cN_{1},\cE)\oplus\chi^{-}(-\Ad_{0})\SobH{\frac{1}{2}}(\cN_{2},\cE)\\
=\chi^{-}(\Ad_{0})\SobH{\frac{1}{2}}(\cN,\cE)\oplus\chi^{+}(\Ad_{0})\SobH{\frac{1}{2}}(\cN,\cE) 
=\SobH{\frac12}(\cN,\cE).
\end{multline*}
That means that, in this special situation, $\SobH{\frac12}(\cN,\cE)$ is canonically identified to a closed subspace of $\checkH(\Dir)$.
It is for this reason that the subspace $\Bmatch$ has a chance of being an boundary condition. 
This is asserted below.
\end{remark}

\begin{lemma}\label{Lem:MatchingEll}
$\Bmatch$ is an elliptically regular boundary condition.
\end{lemma}

\begin{proof}
The full assertion, including that $\Bmatch$ is a boundary condition, can be obtained if we can write it as a graph as in Definition~\ref{Def:GraphDecomp} and by invoking Theorem~\ref{Thm:EllDecomp}.

In light of Proposition~\ref{Prop:Split}, define:
\begin{align*}
V_{-} & :=\chi^{-}(\Ad_{0})\Lp{2}(\cN,\cE)\oplus\chi^{+}(\Ad_{0})\Lp{2}(\cN,\cE) ,
\\
V_{+} & :=\chi^{+}(\Ad_{0})\Lp{2}(\cN,\cE)\oplus\chi^{-}(\Ad_{0})\Lp{2}(\cN,\cE) ,
\\
W_{\pm} & :=\set{ 0} .
\end{align*}
Moreover, define $g: V_{-}\to V_{+}$ by
\[ g=
\begin{pmatrix}
&\id
\\
\id &
\end{pmatrix} .
\]
Then,
\begin{multline*}
\graph( g\rest{\SobH{\frac12}} )  =\set{
\begin{pmatrix}
u
\\
v
\end{pmatrix}
+g
\begin{pmatrix}
u
\\
v
\end{pmatrix}
: u,v\in\SobH{\frac{1}{2}}(\cN,\cE) }
\\
 =\set{
\begin{pmatrix}
u
\\
v
\end{pmatrix}
+
\begin{pmatrix}
v
\\
u
\end{pmatrix}
: u,v\in\SobH{\frac{1}{2}}(\cN,\cE) }
 =\set{
\begin{pmatrix}
u+v
\\
u+v
\end{pmatrix}
: u,v\in\SobH{\frac{1}{2}}(\cN,\cE) }
 =\Bmatch.
\end{multline*}
By Theorem~\ref{Thm:EllDecomp}, we have that $\Bmatch$ is an elliptically regular boundary condition.
\end{proof}

\begin{remark}
Historically, attention has been focused on  elliptically regular boundary conditions  that can be  obtained as ranges of pseudo-differential operators of order zero acting on $\SobH{\frac12}(\partial\cM,\cE)$.
Pseudo-differential operators of order zero are pseudo-local.
By definition of $\Bmatch$, information at $\cN_1$ is matched to that of $\cN_2$.
Since $\cN_1$ and $\cN_2$ can be very far from each other, pseudo-locality is precluded for the matching condition and hence cannot be obtained as a range of a pseudo-differential projector of order zero. 
This demonstrates the power and usefulness of the graphical decomposition.
It is also a fundamental and important observation that lead to the development of the graphical decomposition in \cite{BB12}. 
\end{remark}

Using the results for deformation in the previous section, we now reduce the matching condition to the APS condition. 

\begin{lemma}\label{Lem:MatchingDeform}
Let $\Ad_{0}$ be an invertible bisectorial adapted boundary operator on $\cN$ and $\Ad :=\Ad_{0}\oplus( -\Ad_{0}) $, the induced invertible bisectorial adapted boundary operator on $\partial\cM =\cN_{1}\sqcup\cN_{2}$.
Then
\[
\Baps(\Ad) = \chi^-(\Ad) \SobH{\frac12}(\partial \cM, \cE) = \SobH{\frac{1}{2}}(\cN,\cE)
\]
and
\[
\ind(\Dir_{\Bmatch}) =\ind(\Dir_{\Baps(\Ad) }) .
\]
\end{lemma}

\begin{proof}
We have
\begin{align*}
\Baps(\Ad) & =\chi^{-}(\Ad)\SobH{\frac{1}{2}}(\partial\cM,\cE)
\\
& =\chi^{-}(\Ad_{0}\oplus( -\Ad_{0}) )\SobH{\frac{1}{2}}(\cN_{1}\sqcup\cN_{2},\cE)
\\
& =\chi^{-}(\Ad_{0})\SobH{\frac{1}{2}}(\cN_{1},\cE)\oplus\chi^{-}( -\Ad_{0})\SobH{\frac{1}{2}}(\cN_{2},\cE)
\\
& =\chi^{-}(\Ad_{0})\SobH{\frac{1}{2}}(\cN,\cE)\oplus\chi^{+}(\Ad_{0})\SobH{\frac{1}{2}}(\cN,\cE)
\\
& =\SobH{\frac{1}{2}}(\cN,\cE) .
\end{align*}
Now, for the choices of $V_{\pm}$ used in Lemma~\ref{Lem:MatchingEll}, i.e., $V_{\pm}:=\chi^{\pm}(\Ad_{0})\Lp{2}(\cN_{1},\cE)\oplus\chi^{\mp}(\Ad_{0})\Lp{2}(\cN_{2},\cE) $ and for $g$ as defined there, let
\[
\Bmatch^{s}:=\graph( sg\rest{\SobH{\frac{1}{2}}}) .
\]
Clearly this is a continuous deformation of $\Bmatch$ to $\Bmatch^{0}=\Baps(\Ad) $ and so by Corollary~\ref{Cor:DeformBC}, the conclusion follows.
\end{proof}

Combining these results, we can now prove the following generalised splitting theorem.

\begin{theorem} [Splitting theorem]\label{Thm:Splitting}
Let $(\cM^{\prime},\mu')$  be a connected boundaryless measured manifold  and $(\cE',\mh^{\cE'}),(\cF', \mh^{\cF'}) \to\cM'$ be Hermitian bundles, carrying a first-order differential operator $\Dir':\Ck{\infty}(\cM',\cE')\to\Ck{\infty}(\cM',\cF')$ which is complete and coercive at infinity.
Suppose that $\cN$ is a two-sided hypersurface and let $\cM,\mu,\cE,\cF,\Dir$, denote the induced objects be as above obtained from cutting $\cM'$ along $\cN$. 
Let $B_1 \subset \SobH{\frac12}(\cN_1,E)$ and $B_2 \subset \SobH{\frac12}(\cN_2,E)$  be a closed subspaces. 
Identifying $\SobH{\frac{1}{2}}(\cN_{i},E) $ with $\SobH{\frac{1}{2}}(\cN,\cE) $, assume that $B_1, B_2 \subset \SobH{\frac12}(\cN,E)$ are complementary subspaces satisfying
\[
B_{1}\oplus B_{2}=\SobH{\frac{1}{2}}(\cN,\cE),
\]
where $\oplus$ is the internal direct sum in $\SobH{\frac12}(\cN,\cE)$. 
Noting $\Baps(\Ad) = \chi^-(\Ad) \SobH{\frac12}(\partial \cM,\cE) = \SobH{\frac12}(\cN,\cE) \subset \checkH(\Dir)$, we have
\[
\ind(\Dir') =\ind(\Dir_{B_{1}\oplus B_{2}}) .
\]

\end{theorem}

\begin{proof}
Let $\Bmatch $ be the matching condition as defined in Definition~\ref{Def:Matching}.
On identifying the pullback sections, say via a map $\Phi$ from $\cE'$ to $\cE$ , we get
\[
\dom(\Dir_{\Bmatch}\circ\Phi) =\dom(\Dir') .
\]
By the assumption that $\Dir'$ and $(\Dir')^\dagger$ are complete and coercive at infinity, by Lemma~\ref{Lem:CoerInherited}, the same holds for the induced $\Dir$ and $\Dir^\dagger$. 
Moreover, it is easy to see that 
\[
\ind(\Dir_{\Bmatch}) =\ind(\Dir') .
\]
From Lemma~\ref{Lem:MatchingDeform},
\[
\ind(\Dir_{\Bmatch}) =\ind(\Dir_{\SobH{\frac12}(\cN,\cE)}) =\ind(\Dir_{B_{1}\oplus B_{2}}), 
\]
which is the required conclusion.
\end{proof}

\begin{remark}
This result is a direct generalisation of the result in the case when $\Dir$ admits a self-adjoint adapted boundary operator $\Ad$ on $\cN$.
We can let $B_{1}\subset\SobH{\frac{1}{2}}(\cN,\cE) $ be elliptically regular on $\cN_{1}$ and $B_{2}:=B_{1}^{\perp,\Lp{2}}\cap\SobH{\frac{1}{2}}(\cN,\cE) $ considered as a boundary condition on $\cN_{2}$.

Now choose $\Ad_{0}$ invertible self-adjoint adapted boundary operator, which can always be obtained by subtracting a sufficiently small number from the self-adjoint boundary adapted operator whose existence we assumed.
Set
\begin{align*}
V_{-} & :=V_{-,1}\oplus V_{-,2},
\\
V_{+} & :=V_{+,1}\oplus V_{+,2},
\\
V_{\pm1} & :=\chi^{\pm}(\Ad_{0})\Lp{2}(\cN,\cE)
\\
V_{\pm,2} & :=\chi^{\pm}( -\Ad_{0})\Lp{2}(\cN,\cE) =\chi^{\mp}(\Ad_{0})\Lp{2}(\cN,\cE) .
\end{align*}
Write
\begin{align*}
B_{1} & =W_{+,1}\oplus\graph( g_{1}: V_{-,1}\to V_{+,1})\cap\SobH{\frac{1}{2}}(\cN,\cE) ,
\\
B_{2} & =W_{+,2}\oplus\graph( g_{2}: V_{-,2}\to V_{+,2})\cap\SobH{\frac{1}{2}}(\cN,\cE) .
\end{align*}
But since $B_{1}\perp B_{2}$ in $\Lp{2}$, we have that $V_{\pm,2}=V_{\mp1}$, $W_{\pm,2}=W_{\mp,1}$ and $g_{2}=-g_{1}^{\ast}$.

The adjoint boundary condition for $B_{1}\oplus B_{2}$ is $B_{2}\oplus B_{1}$.
Therefore, $B_1\oplus B_2 =\SobH{\frac12}(\cN,\cE)$ and so Theorem~\ref{Thm:Splitting} applies.
\end{remark} 

\begin{definition}[Separation]
\label{Def:Separate}
Suppose $\cN\subset\cM'$ is a two-sided hypersurface and that  by cutting along $\cN$, we obtain
$\cM' =\cM_1\disunion\cM_2$ with $\partial\cM_1  =\partial\cM_2 =\cN$.
Then we say that $\cN$ separates $\cM'$ (into $\cM_1$ and $\cM_2$).
\end{definition}

As before, we obviously obtain induced objects $\cE_i,\cF_i,\Dir_i$ et cetera via pullback to $\cM_i$. 
The splitting theorem applies in a particular and useful way when $\cN$ \emph{separates} $\cM^{\prime}$, leading to a decomposition theorem.

\begin{corollary}[Decomposition]
\label{Cor:Decomp}
Assume the hypothesis of Theorem~\ref{Thm:Splitting}. 
In addition, assume that $\cN$ separates $\cM'$ into $\cM_1$ and $\cM_2$, and $\Ad_{0}$ be an invertible bisectorial adapted boundary operator on $\cN$ pointing into $\partial\cM_1$.
Let
\begin{align*}
B_{1} &:=\chi^{-}(\Ad_{0})\SobH{\frac{1}{2}}(\partial\cM_1,E) =\chi^{-}(\Ad_{0})\SobH{\frac{1}{2}}(\cN,\cE) ,\quad\text{and}\\
B_{2} &:=\chi^{-}( -\Ad_{0})\SobH{\frac {1}{2}}(\partial\cM_2,E)=\chi^{+}(\Ad_{0})\SobH{\frac{1}{2}}(\cN,\cE). 
\end{align*}
Then,
\[
\ind(\Dir) =\ind(\Dir_{1, B_{1}}) +\ind(\Dir_{2,B_{2}}).
\]
\end{corollary}
\begin{proof}
This is immediate from  $\Dir' =\Dir_1\oplus\Dir_2$ and invoking Theorem~\ref{Thm:Splitting}.
\end{proof}

\section{The Phi and relative index theorems}

In this section, we prove the $\Phi$ and relative index theorems.
For that, we will fix two measured manifolds $(\cM_1,\mu_1)$ and $(\cM_2,\mu_2)$.
These will be equipped with Hermitian vector bundles $(\cE_i,\mh^{\cE_i})\to\cM_i$ and $(\cF_i,\mh^{F_i})\to\cM_i$ carrying first-order elliptic  differential operators $\Dir_i :\Ck{\infty}(\cM_i,\cE_i)\to\Ck{\infty}(\cM_i,\cF_i)$.

\begin{definition}[Agree outside closed subset]
\label{Def:Agree}
Let $\cM_{1},\cM_{2}$ be manifolds and $(\cE_{i},h^{\cE_{i}}) ,(\cF_{i},h^{\cF_{i}})\to\cM_{i}$ Hermitian vector bundles.
Let $\Dir_{i}:\Ck{\infty}(\cM_i,\cE_{i})\to\Ck{\infty}(\cM_i,\cF_{i})$ and $\cK_{i}\subset\cM_{i}$ a closed subset.
Then we say that $\Dir_{1}$ and $\Dir_{2}$ agree outside $\cK_{1},\cK_{2}$ if they are related by vector bundle isometries $\cE_{1}\rest{\cM_{1}\setminus\cK_{1}}\cong\cE_{2}\rest{\cM_{2}\setminus\cK_{2}}$ and $\cF_{1}\rest{\cM_{1}\setminus\cK_{1}}\cong\cF_{2}\rest{\cM_{2}\backslash\cK_{2}}$
satisfying  the following.

\begin{enumerate}[label=(\roman*)]
\item
There is a diffeomorphism $f:\cM_{1}\setminus\cK_{1}\to\cM_{2}\setminus\cK_{2}$.

\item
There exist vector bundle isometries
\[
I_{\cE}:\cE_{1}\rest{\cM_{1}\setminus\cK_{1}}\to\cE_{2}\rest{\cM_{2}\setminus\cK_{2}}
\quad\text{and}\quad
I_{\cF}:\cF_{1}\rest{\cM_{1}\setminus\cK_{1}}\to\cF_{2}\rest{\cM_{2}\setminus\cK_{2}}
\]
over $f$.
That is, $I_\cE$ and $I_\cF$ are fibrewise linear isometries such that\ the following diagrams commute:
\[
\begin{array} [c]{ccc}
\cE_{1}\rest{\cM_{1}\setminus\cK_{1}} &\overset{I_{\cE}}{\longrightarrow} &\cE_{2}\rest{\cM_{2}\setminus\cK_{2}}
\\
\downarrow & &\downarrow
\\
\cM_{1}\setminus\cK_{1} &\overset{f}{\longrightarrow} &\cM_{2}\setminus\cK_{2}
\end{array}
\quad\text{and}\quad
\begin{array} [c]{ccc}
\cF_{1}\rest{\cM_{1}\setminus\cK_{1}} &\overset{I_{\cF}}{\longrightarrow} &\cF_{2}\rest{\cM_{2}\setminus\cK_{2}}
\\
\downarrow & &\downarrow
\\
\cM_{1}\setminus\cK_{1} &\overset{f}{\longrightarrow} &\cM_{2}\setminus\cK_{2}.
\end{array}
\]

\item The operators $\Dir_1$ and $\Dir_2$ are related by $I_\cE, I_\cF$ and $f$.
Explicitly,
\[
I_{\cF}\circ (\Dir_{1}u)\circ f^{-1}=\Dir_{2}( I_{\cE}\circ u\circ f^{-1})
\]
for all $u\in\Ck{\infty}(\cM_{1}\setminus\cK_{1},\cE_{1}) $.
\end{enumerate}
\end{definition}

\begin{theorem}[$\Phi$-relative index theorem]
\label{Thm:PhiRelInd}
Let $(\cM_{1},\mu_1)$, $(\cM_{2},\mu_2)$ be measured manifolds without boundary and $(\cE_{i},\mh^{\cE_{i}}),(\cF_{i},\mh^{\cF_{i}})\to\cM$ Hermitian bundles with $\Dir_{i}:\Ck{\infty}(\cM_i,\cE_i)\to\Ck{\infty}(\cM_i,\cF_i)$, first-order elliptic and coercive at infinity.

Suppose the following: 
\begin{enumerate}[label=(\roman*)]
\item  $\cK_i\subset\cM_i$ such that $\Dir_1,\Dir_2$ agree outside of $\cK_1,\cK_2$ as in in Definition~\ref{Def:Agree};
\item the densities  $\mu_i$ satisfy $\mu_{1}=f^{\ast}\mu_{2}$ on $\cM_1\setminus\cK_1$;
\item there exists a compact two-sided hypersurfaces $\cN_{1}$ separating $\cM_{1}=\cM_{1}^{\prime}\cup\cM_{1}^{\prime\prime}$ with $\partial\cM_{1}^{\prime}=\partial\cM_{1}^{\prime\prime}=\cN_{1}$ and $\cK_{1}\subset\interior\cM_{1}^{\prime}$.
\end{enumerate}

Then, $\cN_2 := f(\cN_1)$ separates $\cM_2 =\cM_{2}^{\prime}\cup\cM_{2}^{\prime\prime}$ with $\cK_2\subset\interior\cM_2'$.

Denote the induced operators on $\cM_i'$ and $\cM_i''$ from $\Dir_i$ by $\Dir_{i}^{\prime}$ and $\Dir_{i}^{\prime\prime}$ respectively and fix an invertible bisectorial adapted boundary operator $\Ad$ to $\Dir_{1}^{\prime}$ on $\cM_{1}^{\prime}$.
Let $B_{1}:=\chi^{-}(\Ad)\SobH{\frac{1}{2}}(\partial\cM_{1},\cE_{1}) $ and $B_{2}$ be identified with $B_{1}$ under $( I_{E},I_{F},f) $.
Then $\Dir_{i},\Dir_{i,B_{i}}^{\prime}$ are Fredholm operators and
\[
\ind(\Dir_{1}) -\ind(\Dir_{2}) =\ind(\Dir_{1,B_{1}}^{\prime}) -\ind(\Dir_{2,B_{2}}^{\prime}) .
\]

\end{theorem}
\begin{proof}
Since $\cM_1\setminus\cK_1$ is diffeomorphic to $\cM_2\setminus\cK_2$, it is clear that $\cN_2 = f(\cN_1)$ separates $\cM_2 =\cM_{2}^{\prime}\cup\cM_{2}^{\prime\prime}$ and $\cK_2\subset\interior\cM_2'$.

By Proposition~\ref{Prop:Split}, since $\Ad$ is an invertible bisectorial adapted boundary operator on $\partial\cM_{1}^{\prime}$, we obtain that $-\Ad$ is an invertible bisectorial adapted boundary operator on $\partial\cM_{1}^{\prime\prime}$.
On setting $B_{1}^{\prime}:=B_{1}$
\[
B_{1}^{\prime\prime}:=\chi^{+}(\Ad)\SobH{\frac{1}{2}}(\partial\cM_{1}^{\prime\prime},E) =\chi^{+}(\Ad)\SobH{\frac {1}{2}}(\cN_1,E)
\]
from Corollary~\ref{Cor:Decomp}, we obtain
\[
\ind(\Dir_{1}) =\ind(\Dir_{1,B_{1}^{\prime}}^{\prime}) +\ind(\Dir_{1,B_{1}^{\prime\prime}}^{\prime\prime}) .
\]

Let $B_{2}^{\prime}:=B_{2}$ and $B_{2}^{\prime\prime}$ be the boundary condition $B_{2}^{\prime\prime}$ pulled across to $\partial\cM_{2}^{\prime\prime}$ via $( I_{E},I_{F},f) $, we get
\[
\ind(\Dir_{2}) =\ind(\Dir_{2,B_{2}^{\prime}}^{\prime}) +\ind(\Dir_{2,B_{2}^{\prime\prime}}^{\prime\prime}) .
\]
Since $\Dir_{1}$ and $\Dir_{2}$ agree outside $\cK_{1}$ and $\cK_{2}$, and $\mu_1 = f^\ast \mu_2$ on $\cM_1 \setminus \cK_1$, 
\[
\ind(\Dir_{1,B_{1}''}^{\prime\prime}) =\ind(\Dir_{2,B_{2}''}^{\prime\prime}) .
\]
By taking the difference we obtain the conclusion.
\end{proof}

To obtain the relative index theorem, we  want to express the right hand side of this index theorem in terms of  the local quantities $\alpha_{0,\Dir_i}$ as described in the hypothesis of this theorem.
These quantities are access by  embedding the manifold $\cM_i'$ inside a larger closed  manifold and appropriately extending $\Dir_i$.

\begin{lemma}\label{Lem:SmartDouble}
Let $\cM_{1}$ and $\cM_{2}$ be two compact manifolds with boundary such that\ $\Dir_{1},\Dir_{2}$ elliptic agree outside $\cK_{i}\subset\interior\cM_{i}$.
Then there are $\tilde{\cM}_{i}$ compact with $\partial\tilde{\cM}_{i}=\emptyset$ such that\ the following hold.

\begin{enumerate}[label=(\roman*)]
\item
$\cM_{i}\subset\tilde{\cM}_{i}$,

\item
$\cE_{i}\subset\tilde{\cE}_{i}$, $\mh^{\tilde{\cE}_{i}}$ smooth with $\mh^{\tilde{\cE}_{i}}\rest{\cM_{i}}=\mh^{\cE_{i}}$,

\item
$\tilde{\Dir}_{i}$ elliptic such that\ $\Dir_{i}$ and $\tilde{\Dir}_{1}$ and $\tilde {D}_{2}$ agree outside $\cK_{1},\cK_{2}$,

\item
$\tilde{\Dir}_{i}\rest{\cK_{i}}=\Dir_{i}\rest{\cK_{i}}$.
\end{enumerate}
\end{lemma}

\begin{proof}
Take $\cM_{i}$ and set $\cM_{1}^{\mathrm{2nd}} :=\cM_1$ as the second copy of $\cM_1$.
We will glue this second copy $\cM_1^{\mathrm{2nd}}$ to $\cM_i$, regardless of whether $i=1$ or $i=2$.
That is, define
\[
\tilde {M}_{i}:=\cM_{i}\cup_{U_i}\cM_{1}^{\mathrm{2nd}}
\]
identifying inside $U_i :=\cM_i\setminus\cK_i$, which by hypothesis is identified with $\cM_1\setminus\cK_1$ via a diffeomorphism.

By hypothesis, the $\Dir_{i}$ agree on $U_{i}$ open and containing $\partial\cM_{i}$.
In $U_{i}$, $\Dir_{1}$ and $\Dir_{2}$ agree through identification by $( I_{\cE},I_{\cF},f) $.
So on doubling, we keep the smooth coefficients and obtain $\tilde{\Dir}_{i}$.

Similarly, $\mh^{\tilde{\cE}_{i}}$ have smooth coefficients also.
Therefore, the conclusions as stated follow.
\end{proof}

With this at our aid, we now prove the relative index theorem in the context of a general first-order elliptic differential operator, extending Theorem~1.21 in \cite{BB12} and Theorem~4.18 in \cite{GL}

\begin{proof}[Proof of Theorem~\ref{Thm:RelInd}] 
Since $\cK_{i}$ are compact, applying Corollary~\ref{Cor:CoerFred} to a manifold without boundary gives that $\Dir_{i}$ are Fredholm if and only if $\Dir_{i}$ are coercive at infinity, say with respect to\ a set $\tilde{\cK}_{i}$ (c.f Definition~\ref{Def:Coercive}).
But $\cK_{i}\cup\tilde{\cK}_{i}$ is still compact and $\Dir_{i}$ are coercive at infinity with respect to\ $\cK_{i}\cup\tilde{\cK}_{i}$ and hence it is easy to see that $\Dir_{1}$ is Fredholm  if and only if  $\Dir_{2}$ is Fredholm since $\Dir_{1},\Dir_{2}$ agree outside the compact set $\cK_{i}\cup\tilde{\cK}_{i}$.

Now from Theorem~\ref{Thm:PhiRelInd}, we obtain $\cN_2 = f(\cN_1)$ separates $\cM_2 =\cM_2'\union\cM_2''$ with $\cK_2\subset\interior\cM_2'$.
Let $\Dir_i'$ and $\Dir_i''$ be the induced operators on $\cM_i'$ and $\cM_i''$ respectively.
Choosing $B_{1}:=\chi^{-}(\Ad)\SobH{\frac{1}{2}}(\cN,\cE) $ for an invertible bisectorial adapted boundary operator $\Ad$ on $\cN_1$, and with $B_2$ pulled back to $\cN_2 = f(\cN_1)$ via $(I_\cE,I_\cF,f)$, we obtain from Theorem~\ref{Thm:PhiRelInd} that
\[
\ind(\Dir_{1}) -\ind(\Dir_{2}) =\ind(\Dir_{1,B_{1}}^{\prime}) -\ind(\Dir_{2,B_{2}}^{\prime}) .
\]

On application of Lemma~\ref{Lem:SmartDouble}, we obtain $\tilde{\cM}_{i}$ containing $\cM_i'$ and $\tilde{\Dir}_{i}$.
But $\tilde {M}_{i}$ is closed, $\tilde{\Dir}_{i}$ is elliptic on a closed manifold, so by Corollary~\ref{Cor:Decomp},
\[
\ind(\tilde{\Dir}_{i}) =\ind(\tilde{\Dir}_{i,B_{i}}^{\prime}) +\ind(\tilde {D}_{i,B_{i}^{\prime\prime}}^{\prime\prime}) ,
\]
where $B_{1}^{\prime\prime}:=\chi^{+}(\Ad)\SobH{\frac{1}{2}}(\cN,\cE) $ and $B_2''$ is $B_1''$ identified on $\tilde{\cM}_2''$ through $(\tilde{I}_E,\tilde{I}_F,\tilde{f})$ given by Lemma~\ref{Lem:SmartDouble}.

By construction, $\tilde{\Dir}_{1}$ and $\tilde{\Dir}_{2}$ agree outside of $\cK_{1},\cK_{2}$.
In particular, $\tilde{\Dir}_{1}$ and $\tilde{\Dir}_{2}$ agree on $\tilde{\cM}_{1}^{\prime\prime}$ and $\tilde{\cM}_{2}^{\prime\prime}$.
Therefore
\[
\ind(\tilde{\Dir}_{1,B_{1}^{\prime\prime}}^{\prime\prime }) =\ind(\tilde{\Dir}_{2,B_{2}^{\prime\prime}}^{\prime\prime}) ,
\]
and
\[
\ind(\Dir_{1,B_{1}}) -\ind(\Dir_{2,B_{2}^{\prime}}) =\ind(\tilde{\Dir}_{1}) -\ind(\tilde{\Dir}_{2}) .
\]
Since $\tilde{\cM}_{i}$ are closed and $\tilde{\Dir}_i$ are first-order elliptic, we can apply Atiyah-Singer index theorem to obtain
\begin{align*}
\ind(\tilde{\Dir}_{i}) & =\int_{\tilde{\cM}_{i}}\alpha_{0,\tilde{\Dir}_{i}}  
=
\int_{\cK_{i}}\alpha_{0,\tilde{\Dir}_{i}}\ d\mu_i  +\int_{\tilde{\cM}_{i}\setminus\cK_{i}}\alpha_{0,\tilde{\Dir}_{i}}\ d\mu_i,
\end{align*}
where $\alpha_{0,\tilde{\Dir}_i}$ is the constant term in asymptotic expansion of 
$$\tr \cbrac{\e^{-\tau\tilde{\Dir}_i^\ast\tilde{\Dir}_i} -\e^{-\tau\tilde{\Dir}_i\tilde{\Dir}_i^\ast}}(x) \sim \sum_{k \geq -n} t^{\frac{k}{2}} \alpha_{k, \tilde{\Dir_i}}(x)$$
as $t \to 0$.
See Theorem~(EIII) and the accompanying Remark 2) in \cite{ABP}.

The operators $\Dir_{1},\Dir_{2}$ agree outside $\cK_{1},\cK_{2}$ and so certainly $\tilde{\Dir}_{1},\tilde{\Dir}_{2}$ agree outside $\cK_{1},\cK_{2}$.
Moreover, since we assume $\mu_2 = f^\ast\mu_1$,
\[
\int_{\tilde{\cM}_{1}\setminus\cK_{1}}\alpha_{0,\tilde{\Dir}_{1}}\ d\mu_1 =\int_{\tilde{\cM}_{2}\setminus\cK_{2}}\alpha_{0,\tilde{\Dir}_{2}}\ d\mu_2.
\]
Also, $\tilde{\Dir}_{i}=\Dir_{i}$ on $\cK_{i}$, so
\[\alpha_{0,\tilde{\Dir}_{i}}( x) =\alpha_{0,\Dir_{i}}( x)\]
for $x\in\cK_{i}$.
Hence, we conclude
\[
\ind(\tilde{\Dir}_{1}) -\ind(\tilde{\Dir}_{2}) =\int_{\cK_{1}}\alpha_{0,\Dir_{1}}\ d\mu_1 -\int_{\cK_{2}}\alpha_{0,\Dir_{2}}\ d\mu_2.
\qedhere
\]
\end{proof}

\begin{remark}
If we were to dispense with using Lemma~\ref{Lem:SmartDouble}, we would need to directly compute $\ind(\Dir_{i,B_{i}}^{\prime})$.
In the situation where $\Ad_1 :=\Ad$ is self-adjoint, we can accomplish this using the Atiyah-Patodi-Singer index theorem.
By pulling across $\Ad$ to $\cM_2'$, to obtain $\Ad_2$ we have that $\dim\ker\Ad_2 =\dim\ker\Ad_1$ and $\eta_{\Ad_1}(s) =\eta_{\Ad_2}(s)$.
Therefore, these terms cancel, and we obtained the stated relative index formula.
The advantage of using the Atiyah-Singer index theorem  as opposed to  the Atiyah-Patodi-Singer index theorem is that we do not require additional assumptions on the $\Dir_i$.
\end{remark}

\begin{remark}
As noted in \cite{BB12}, when the operators $\Dir_i$ are Dirac-type on Riemannian manifold $(\cM_i, \mg_i)$, the quantities $\alpha_{0,\Dir_i}$ can be computed pointwise and explicitly in local coordinates from the coefficients of $\Dir_i$ and their derivatives (c.f. Chapter 4 in \cite{BGV}).
However, as we have aforementioned, the Rarita-Schwinger operator is a naturally determined operator from physics and geometry which falls outside of the Dirac-type regime.  
For this and other geometric operators, understanding these quantities geometrically as for the Dirac-type case, beyond the abstract description arising from the asymptotic expansion, is likely to be interesting and insightful.
These are open questions in the field which are current and actively pursued. 
\end{remark}


\bibliographystyle{alpha}

\setlength{\parskip}{0pt}
\end{document}